\documentclass[11pt]{article}
\usepackage[utf8]{inputenc}
\usepackage{authblk}
\usepackage{mathtools}
\usepackage[english]{babel}
\usepackage{amsmath}
\usepackage{amsthm}
\usepackage{amsfonts}
\usepackage{amssymb}
\usepackage{graphicx}
\usepackage[usenames,dvipsnames]{color}
\usepackage{hyperref}
\usepackage[left=2cm,right=2cm,top=2cm,bottom=2cm]{geometry}
\usepackage{dsfont}
\usepackage{pgf,tikz,pgfplots}
\usepackage{enumerate}

\usepackage[normalem]{ulem}
\theoremstyle{plain}
\newtheorem{theorem}{Theorem}[section]
\newtheorem{corollary}[theorem]{Corollary}

\newtheorem{lemma}[theorem]{Lemma}

\newtheorem{conjecture}[theorem]{Conjecture}

\theoremstyle{definition}

\newtheorem{remark}[theorem]{Remark}

\newtheorem{def/prop}[theorem]{Definition/Proposition}

\newcommand{\ind}{\mathds{1}}

\pgfplotsset{compat = 1.16}%take out before submission

\DeclareMathOperator{\Prob}{\mathbb{P}}

\DeclareMathOperator{\eps}{\varepsilon}
\DeclareMathOperator{\cA}{\mathcal{A}}

\DeclareMathOperator{\cE}{\mathcal{E}}

\DeclareMathOperator{\cN}{\mathcal{N}}

\DeclareMathOperator{\cP}{\mathcal{P}}

\title{\scshape $d$-connectivity of the random graph with restricted budget}

\author{Lyuben Lichev\footnote{E-mail: lyuben.lichev@univ-st-etienne.fr}}
\affil{Univ. Jean Monnet, Saint-Etienne, France}

\begin{document}

\maketitle

\begin{abstract}
In this short note, we consider a graph process recently introduced by Frieze, Krivelevich and Michaeli. 
In their model, the edges of the complete graph $K_n$ are ordered uniformly at random and are then revealed consecutively to a player called Builder. 
At every round, Builder must decide if they accept the edge proposed at this round or not. 
We prove that, for every $d\ge 2$, Builder can construct a spanning $d$-connected graph after $(1+o(1))n\log n/2$ rounds by accepting $(1+o(1))dn/2$ edges with probability converging to 1 as $n\to \infty$. 
This settles a conjecture of Frieze, Krivelevich and Michaeli.
\end{abstract}
\noindent
\textbf{MSC class:} 05C80, 05C40

\section{Introduction}

The random graph process introduced by Erd\H{o}s and R\'enyi~\cite{ER59, ER60} is currently one of the most well-understood discrete stochastic processes. 
Several interesting variations of this process have been studied in the literature: examples include the classic $H$-free process~\cite{BK10,ESW95,OT01,War11} with special focus put on the triangle-free process~\cite{Boh09,BK21,FPGM20} as well as the recently analysed model of multi-source invasion percolation on the complete graph~\cite{A-BB23,LMP23}. 
Also recently, Frieze, Krivelevich and Michaeli~\cite{FKM22} introduced the following version: the edges of the complete graph $K_n$ are ordered uniformly at random and are then proposed one at a time to a player called \emph{Builder}. 
At first, Builder starts with an empty graph. 
At every round, Builder must decide if they want to add the currently proposed edge to their graph based only on the sequence of edges proposed up to now. 
Builder's goal is to reach a configuration with some graph property $\cP$. 
However, the player is allowed to wait for at most $t$ rounds, and moreover, they have the right to accept at most $b$ edges.

\paragraph{The random graph process on restricted budget: known results.} 
Before presenting the main result of this work, we briefly survey known results on this new model from~\cite{Ana22,FKM22}.
To do this, recall that a sequence of events $(E_n)_{n = 1}^{\infty}$ holds \emph{asymptotically almost surely} (which we abbreviate \emph{a.a.s.}) if $\mathbb P(E_n)\to 1$ as $n\to \infty$. 

Most of the results can be roughly divided into two types.
The first type corresponds to results where $t$ is the hitting time of a property $\mathcal P$ by the Erd\H{o}s-R\'enyi graph process and $b$ is a constant factor away from the minimum budget needed to ensure the property $\mathcal P$ deterministically.
Such results include:
\begin{itemize}
    \item Theorem~1 from~\cite{FKM22}, showing that a.a.s.\ Builder can attain a graph of minimum degree $d\ge 1$ at the hitting time $\tau_d$ for this property under the limitation $b\ge o_d n + \omega(\sqrt{n}\log n)$ for some constant $o_d\in (d/2, 3d/4]$. Similar result holds for $d$-connectivity for all $d\ge 3$.
    \item Theorem~3 from~\cite{FKM22}, showing that a.a.s.\ Builder can form a Hamiltonian cycle at the hitting time $\tau_2$ given that $b\ge C n$ for some constant $C > 1$.
\end{itemize}
We remark that the authors of~\cite{FKM22} characterise the constant $o_d$ explicitly and conjecture that a.a.s.\ Builder cannot construct a graph with minimum degree $k$ at the hitting time if $b\le (o_d-\eps)n$ for any $\eps > 0$.
This conjecture was recently disproved by Katsamaktsis and Letzter~\cite{KL24}.

The second type of results focus on constructing a graph with a property $\mathcal P$ where both $t$ and $b$ have asymptotically optimal values.
\begin{itemize}
    \item Theorem~2 from~\cite{FKM22} shows that, for every $\eps > 0$ and $d\ge 1$, a.a.s.\ Builder can attain a graph of minimum degree $d$ for $t\ge (1+\eps)n\log n/2$ and $b\ge (1+\eps)dn/2$.
    \item Theorem~4 from~\cite{FKM22} shows a similar result for the construction of a perfect matching.
    \item Anastos~\cite{Ana22} showed that the same conclusion holds for the construction of a Hamiltonian cycle.
\end{itemize}
In the light of Anastos' result~\cite{Ana22}, it would be interesting to know if the constant $C$ in Theorem~3 from~\cite{FKM22} could be improved to a $(1+o(1))$-factor.
Finally, related threshold results were shown in~\cite{FKM22} for the construction of trees and cycles of constant size.

\paragraph{The main result.} We start with some basic terminology allowing us to state our main result. For every $s\ge 0$, we denote the graph of all edges accepted by Builder during the first $s$ rounds by $G_s$. Of course, this graph depends on the choices of Builder during the first $s$ rounds. 
A \emph{strategy} of Builder is a (possibly random) function which, given an integer $s\le \tbinom{n}{2}$ and an $s$-tuple $(e_1, \ldots, e_s)$ of distinct edges of $K_n$, 
outputs 1 when Builder accepts the edge $e_s$ at round $s$ given that the edges $e_1, \ldots, e_{s-1}$ were proposed before in this order, and 0 if Builder ignores the edge $e_s$. 
A \emph{$(t,b)$-strategy} of Builder is a strategy that, for every $s\le t$, indicates whether to accept the edge presented at round $s$ based on the edges proposed before under the limitation that $|E(G_s)|\le b$.

In an earlier version of~\cite{FKM22}, the authors stated the following conjecture (now appearing as Theorem~6.1 and attributed to the current work).

\begin{conjecture}[Conjecture~8 in version~1 of~\cite{FKM22}]\label{conj main}
Fix $\eps > 0$ and a positive integer $d\ge 2$. 
If $t\ge (1+\eps) n\log n/2$ and $b\ge (1+\eps)dn/2$, then there exists a $(t,b)$-strategy of Builder such that $G_t$ is a.a.s.\ $d$-connected.
\end{conjecture}

\noindent
Note that the statement of Conjecture~\ref{conj main} does not hold for $d=1$ since Builder needs at least $n-1$ edges to construct a connected graph. 
The current work is dedicated to the proof of Conjecture~\ref{conj main}.

\begin{theorem}\label{thm main}
Fix a positive integer $d\ge 2$. 
Then, for every $\eps > 0$, given $t \ge (1+\eps) n\log n/2$ and $b \ge (1+\eps)dn/2$, there exists a $(t,b)$-strategy of Builder such that a.a.s.\ $G_t$ is $d$-connected.
\end{theorem}

\noindent
In fact, our proof exhibits an appropriate strategy that does not depend on the value of $\eps$.\\

\noindent
\textbf{Outline of the proof of Theorem~\ref{thm main}.} The proof is divided into two parts. The idea when $d\ge 4$ is the following. 
Firstly, we construct consecutively $d$ almost perfect matchings within $s \ll n\log n$ rounds in a closely related model that allows certain repetitions of the edge proposals. 
Next, we show that the resulting graph $G$ a.a.s.\ contains a subgraph $G'$ on $n-o(n)$ vertices and minimum degree $d-1$ which in a sense is close to being $d$-connected. 
Then, we use $(1+o(1)) n\log n/2$ more rounds and $o(n)$ more accepted edges to simultaneously build upon the graph $G'$ and transform it into a $d$-connected graph, and connect each of the vertices outside $G'$ to $G'$ itself via $d$ edges. 

Unfortunately, this approach fails when $d\in \{2,3\}$. 
In this case, we provide an alternative argument. Firstly, we construct a set of long paths. Then, we construct one cycle of length $n-o(n)$ (when $d=2$) and $3n/4-o(n)$ (when $d=3$). 
Finally, we use this long cycle as a ``skeleton'' to which we connect all remaining vertices by 2 edges (when $d=2$) and by 3 edges (when $d=3$). 
This strategy bears a resemblance to the approach in Section~4.2.2 in~\cite{FKM22}.\\

\noindent
\textbf{Notations.} As usual, upper and lower integer parts that are of no importance for the arguments are omitted for better readability. 
In this paper, $\omega = \omega(n)$ is a fixed function satisfying $1\ll \log \omega\ll \log\log n$, and $p = p(n) = \log n/n$. 
Also, for two sets $U\subseteq V$, we write $V\setminus U$ for the set of all elements contained in $V$ but not in $U$.

For a graph $G$, we denote by $V(G)$ the vertex set of $G$ and by $E(G)$ the edge set of $G$. 
Moreover, $\Delta(G)$ stands for the maximum degree of $G$ and, given a vertex $v$ in $G$, we denote by $\deg_G(v)$ the degree of $v$ in $G$. 
For a graph $G$ and a subgraph $H\subseteq G$ or a set of vertices $U\subseteq V(G)$, we write $G\setminus H$ (resp. $G\setminus U$) to denote the graph obtained from $G$ by deleting all vertices in $H$ (resp. in $U$) from $G$.
The neighbourhood of a vertex $v$ in $G$ is denoted by $N_G(v)$ and, given a set of vertices $U\subseteq V(G)$, $N_G(U) = (\bigcup_{v\in U} N_G(v))\setminus U$. 
Note that $\deg_G$ and $N_G$ are simply written $\deg$ and $N$ in case no ambiguity arises. Moreover, for a set $S\subseteq V(G)$, we denote by $G[S]$ the graph induced from $G$ by the set $S$. 
Given a connected graph (or multigraph) $G$, a \emph{cutset} is a set of vertices $U\subseteq V(G)$ satisfying that $G\setminus U$ is a disconnected graph. 
Moreover, $G$ is \emph{$d$-connected} if $|V(G)|\ge d+1$ and there is no cutset of $G$ of size at most $d-1$. 
Finally, the vertices of $K_n$ are denoted $\{w_1,\ldots,w_n\}$.

\section{Preliminaries}

First of all, we recall one instance of the well-known Chernoff's inequality.
\begin{lemma}[see Theorem 2.1 in \cite{JLR00}]\label{lemma chern}
Given a binomial random variable $X$, for every $t\ge 0$,
\begin{equation*}
\Prob(|X - \mathbb E X| \ge t) \le 2\exp \left( - \frac {t^2}{2 (\mathbb E X + t/3)} \right).
\end{equation*}
\end{lemma}

\noindent
Next, recall that for a set $S$ and a real number $p\in [0,1]$, the \emph{binomial random subset} $\mathrm{Bin}(S,p)$ is a subset of $S$ in which all elements of $S$ appear independently with probability $p$. 
The next lemma is a classic comparison result between $\mathrm{Bin}(S,p)$ and the \emph{uniform random $m$-element subset of $S$} denoted $\mathrm{Bin}(S,m)$.

\begin{lemma}[see \cite{L90} and Theorem~1.4 in~\cite{FK16}]\label{lemma comp}
Fix a sequence of sets $(S_n)_{n\ge 1}$ satisfying $|S_n| = n$ for all $n\ge 1$, and functions $m = m(n)$ and $\nu = \nu(n)$ satisfying $m\to \infty$ and $m^{-1/2} \nu\to \infty$ as $n\to \infty$. 
Set $p_- = p_-(n) = \max(0,(m - \nu)/n)$ and  $p_+ = p_+(n) = \min((m + \nu)/n,1)$. 
Then, there is a coupling of the sets $\mathrm{Bin}(S_n,p_-)$, $\mathrm{Bin}(S_n,m)$ and $\mathrm{Bin}(S_n,p_+)$ so that a.a.s.\ $\mathrm{Bin}(S_n,p_-)\subseteq \mathrm{Bin}(S_n,m)\subseteq \mathrm{Bin}(S_n,p_+)$.
\end{lemma}

\begin{remark}\label{rem comp}
Note that both~\cite{L90} and Theorem~1.4 in~\cite{FK16} compare the \emph{binomial random graph} $G(n,p)$, where edges appear independently with probability $p\in [0,1]$, with the original \emph{Erd\H{o}s-R\'enyi random graph} $G(n,m)$, which is a uniformly chosen graph among the graphs with $n$ vertices and $m$ edges. 
However, the proof method extends verbatim in the more general setting of Lemma~\ref{lemma comp}.
\end{remark}

\noindent
The next lemma states and proves a useful property of the Erd\H{o}s-R\'enyi process.

\begin{lemma}\label{lemma eq union bound}
Fix an integer $k\ge 1$, a positive function $\mu = \mu(n)$ satisfying $\tfrac{\log\log n}{\log n}\ll \mu\ll 1$, and positive integer functions $m = m(n)$ and $M = M(n)$ satisfying $m = o(n\log n)$ and $M - m = (1+\mu)n\log n/2$. 
For every $i\in [n]$, fix a set $S_i\subseteq V(K_n)$ (possibly depending on the first $m$ rounds of the Erd\H{o}s-R\'enyi process) of size $|S_i| \ge (1-\mu/2)n$. 
Then, the following event holds a.a.s.: for every $i\in [n]$, at least $k$ edges connecting $w_i$ and $S_i$ are proposed to Builder during the rounds in the interval $[m+1, M]$.
\end{lemma}
\begin{proof}
Recall that $p = \log n/n$ and set $N = \tbinom{n}{2}$. 
Then, conditionally on the event that the random graph $G(n,p)$ has at least $m$ edges, it dominates the graph $G_m$ stochastically with respect to inclusion.
Thus,
\begin{align*}
\Prob(\Delta(G_m)\ge 3\log n)
&\le \Prob(\Delta(G(n,p))\ge 3\log n\mid |E(G(n,p))|\ge m)\\
&\le \frac{\Prob(\Delta(G(n,p))\ge 3\log n)}{\Prob(|E(G(n,p))|\ge m)} \le  \frac{n\Prob(\deg_{G(n,p)}(w_1)\ge 3\log n)}{1-\Prob(|E(G(n,p))| < m)},
\end{align*}
where the last inequality follows from a union bound over all $n$ vertices in $K_n$.
Using Chernoff's inequality for the binomial distributions $\mathrm{Bin}(n-1,p)$ and $\mathrm{Bin}(N,p)$ shows that
\begin{equation*}%\label{eq chernoff}
\begin{split}
\Prob(\Delta(G_m)\ge 3\log n)
&\le (1+o(1)) n\Prob(\mathrm{Bin}(n-1, p)\ge 3\log n)\\
&\le (1+o(1)) 2n\exp\left(-\frac{(3\log n)^2}{2((n-1)p+\log n)}\right) = n^{-5/4+o(1)} = o(1).
\end{split}
\end{equation*}
We condition on the a.a.s.\ event $\{\Delta(G_m)\le 3\log n\}$ and expose the graph $G_m$.

Next, we fix $i\in [n]$ and consecutively expose the edges added at rounds $m+1,\ldots,M$ in the Erd\H{o}s-R\'enyi process. 
We estimate the probability that only $j\in [0,k-1]$ edges connecting $w_i$ and $S_i$ have been proposed to Builder during the rounds $m+1,\ldots,M$.
On the one hand, at each of these rounds, the probability that the new edge goes between the vertex $w_i$ and the set $S_i$ is bounded from above by $\tfrac{n-1}{N-M}$. 
On the other hand, the probability that the edge proposed at round $\ell\in [m+1,M]$ does not go between $w_i$ and $S_i$ is bounded from above by \[1 - \frac{(1-\mu/2)n-j-\Delta(G_m)}{N-\ell+1}\le 1 - \frac{(1-\mu/2)n-j-3\log n}{N-\ell+1}.\]
Hence, the probability that less than $k$ edges connecting $w_i$ and $S_i$ have been proposed to Builder during the rounds $m+1,\ldots,M$ is bounded from above by
\begin{equation}\label{eq union bound}
\begin{split}
&\sum_{j=0}^{k-1} \sum_{J\subseteq [m+1, M];\, |J|=j} \left(\frac{n-1}{N-M}\right)^j\prod_{\ell\in [m+1,M]\setminus J} \left(1 - \frac{(1-\mu/2)n-j-3\log n}{N-\ell+1}\right)\\
=\; 
&O\left(\sum_{j=0}^{k-1} \binom{M-m}{j} \left(\frac{2}{n}\right)^j \prod_{\ell=m+1}^{M}\left(1 - \frac{(2-\mu)n - O(\log n)}{n^2 - O(n\log n)}\right)\right)\\
=\; 
&O\left((\log n)^{k-1} \exp\left(-\left(1+O\left(\frac{\log n}{n}\right)\right)\frac{(M-m)(2-\mu)}{n}\right)\right)\\
=\;
&O\left(\frac{\exp((k-1)\log\log n-(1+o(1))\mu \log n/2)}{n}\right) = o\left(\frac{1}{n}\right),
\end{split}
\end{equation}
where for the third equality we used that $(1+\mu)(2-\mu)/2 = 1+\mu/2+o(\mu)$.
The proof is completed by a union bound over all $n$ vertices of $K_n$.
\end{proof}

\noindent
Our last preliminary result is a simple deterministic lemma providing a way to build $d$-connected graphs from smaller $d$-connected graphs.

\begin{lemma}\label{lem last step}
Fix an integer $d\ge 1$ and two graphs $H\subseteq G$ such that $H$ is $d$-connected and every vertex in $G\setminus H$ is adjacent to at least $d$ vertices in $H$.
Then, $G$ is also a $d$-connected graph.
\end{lemma}
\begin{proof}
Fix any set $U\subseteq V(G)$ of size $|U| = d-1$. Then, since $H$ is a $d$-connected graph, $H\setminus U$ is a connected graph. 
Moreover, every vertex in $G\setminus H$ is connected by an edge in $G\setminus U$ to a vertex in $V(H)\setminus U$. Hence, $G\setminus U$ is a connected graph, which finishes the proof.
\end{proof}

\section{\texorpdfstring{Proof for the case $d\ge 4$}{}}

We continue with the proof of Theorem~\ref{thm main} in the case $d\ge 4$. 
To start with, we define an auxiliary process. 
Next, we describe a strategy of Builder in the auxiliary process that within $o(n\log n)$ rounds constructs a well-connected random graph, which we relate to a configuration model with degree sequence of maximum degree $d$. 
Then, we couple the above strategy with a valid strategy of Builder in the original model, and thus ensure that Builder can a.a.s.\ construct a large graph that is almost $d$-connected in a certain sense within $o(n\log n)$ rounds. 
Finally, by using $(1+o(1))n\log n/2$ additional rounds, we boost this graph to a $d$-connected graph and thus complete the analysis.\\

\noindent
\textbf{Stage 1: the auxiliary process and Builder's strategy.} ($o(n\log n)$ rounds, $dn/2 + o(n)$ accepted edges)\\

\noindent
The auxiliary process is divided into $d$ independent \emph{iterations}, and every iteration consists of the first $n\log n/\omega$ rounds of an Erd\H{o}s-R\'enyi process. 
At every iteration, we let Builder construct a matching by greedily accepting every edge that does not share a vertex with the ones accepted up to now during the current iteration. 
Finally, we colour the edges constructed at iteration $i\in [d]$ in colour $i$; note that one edge may be coloured in more than one colour.

Using that the $d$ matchings are independent, by conditioning on the set of vertices of degree 1 at every iteration, 
we obtain a graph distributed according to a configuration model in which every vertex is incident to at most $d$ half-edges in different colours, and monochromatic half-edges are matched uniformly at random. 
Our motivation to introduce this auxiliary process is rooted in the fact that the configuration model is quite well-understood and easier to work with.

The inconvenience in using the above auxiliary model is that Builder will not be able to construct some of the edges in the $d$ matchings above in the original model.
We call an edge \emph{repeated} if it was proposed at two or more iterations in the auxiliary model.
We note that, in order to compare the strategy of Builder in the auxiliary process and in the original one, it is sufficient to ignore the rounds at which repeated edges are proposed for a second, third, etc. time.

The next lemma provides a convenient way to analyse the set $R$ of repeated edges in the graph $G$ constructed in the auxiliary process: it turns out that this set of edges is dominated by a Bernoulli percolation on $G$ with a suitably chosen parameter.
This will be useful in the third stage of our analysis where not only the number but also the distribution of the edges in $R$ will matter.

\begin{lemma}\label{lemma repeated}
Conditionally on the graph $G$ constructed during the auxiliary process (where only the edges are revealed but not their colours), one may couple the set $R$ of repeated edges in $G$ with a binomial random graph $\widehat G\sim G(n,p)$ with $p = p(n) = \log n/n$ so that a.a.s.\ $R\subseteq E(G)\cap E(\widehat G)$.
\end{lemma}
\begin{proof}
Define $q = q(n) = 3 \log n/\omega n$ and, for every $i\in [d]$, let $S_i$ be the set of edges proposed during the $i$-th stage of the auxiliary process. 
By Lemma~\ref{lemma comp} and a union bound over $i\in [d]$, one may couple $(S_i)_{i=1}^d$ with an i.i.d.\ family of random graphs $(\widehat G_i)_{i=1}^d$ with distribution $G(n, q)$ so that a.a.s.\ $S_i\subseteq E(\widehat G_i)$ for all $i\in [d]$.
Moreover, for every edge $e\in E(K_n)$, the probability that this edge appears in at least two of $(\widehat G_i)_{i=1}^d$ given that it appears in at least one of $(\widehat G_i)_{i=1}^d$ is
\[\frac{1 - (1 - q)^d - dq(1-q)^{d-1}}{1 - (1-q)^d}\le (d-1)q\le p,\]
which concludes the proof.
\end{proof}

\noindent
\textbf{Stage 2: deleting the vertices of degree at most $d-2$.}\\

\noindent
Let us first remark that the second stage does not require any additional rounds or accepted edges. 
Instead, it serves to analyse the graph constructed by Builder in the auxiliary process and prepares the terrain for the additional boosting at Stage 3.
The next lemma roughly says that a.a.s.\ the greedy construction at each of the $d$ iterations of the auxiliary process outputs an almost perfect matching.

\begin{lemma}\label{lemma non max deg}
A.a.s.\ the set of vertices that do not participate in some of the $d$ matchings in the auxiliary process is of size at most $\omega^3 n/\log n$.
\end{lemma}
\begin{proof}
Let us consider the matching constructed at a fixed iteration. Define $s = (n - \omega^2 n/\log n)/2$. For every $i\in [s]$, denote by $T_i$ the time needed to extend the partial matching of size $i-1$ to a matching of size $i$. 
Then, $T_i$ is stochastically dominated by a geometric distribution with parameter $\tbinom{n-2(i-1)}{2}{/}\tbinom{n}{2}$: indeed, the set of $n-2(i-1)$ vertices isolated in the time interval $[T_{i-1}, T_i-1]$ spans $\tbinom{n-2(i-1)}{2}$ edges and the total number of edges yet to be revealed is less than $\tbinom{n}{2}$. Therefore, 
\begin{align*}
\mathbb E\left[\sum_{i=1}^{s} T_i\right] 
&\le\; n(n-1) \sum_{i=1}^{s} \frac{1}{(n-2i+2)(n-2i+1)} = (1+o(1)) \frac{n^2}{2} \sum_{j=1}^{2s} \frac{1}{(n-j)(n-j-1)}\\
&=\; (1+o(1)) \frac{n^2}{2} \left(\frac{1}{n-2s-1} - \frac{1}{n}\right) = (1+o(1))\frac{\omega^2 n}{2\log n}.
\end{align*}
Finally, Markov's inequality implies that, at each of the $d$ iterations, a.a.s.\ $\sum_{i=1}^{s} T_i \le \omega^3 n/\log n$, which completes the proof.
\end{proof}
\noindent
We remark that an alternative proof of the above lemma uses the fact that a.a.s.\ the uniform random subgraph of $K_n$ on $n\log n/\omega$ edges has maximum independent set of size less than $\omega^3 n/\log n$.

At this point, we condition on the a.a.s.\ event of Lemma~\ref{lemma non max deg} and on the coloured vertex degrees, that is, we associate to every vertex in the $i$-th matching a half-edge of colour $i$. 
As the $d$ matchings are chosen uniformly at random and independently of each other, 
the coloured graph constructed in the auxiliary process is distributed according to a configuration model where the half-edges are paired uniformly at random while respecting the colouring. 
From this alternative viewpoint, edges with more than one colour are seen as multiedges with different colours.

Next, we describe an algorithm whose purpose it to find an almost-spanning subgraph with minimum degree $d-1$ where almost all vertices have degree $d$. 
While this graph is not $d$-connected, it will turn out that a.a.s.\ one can add $o(n)$ edges to make it $d$-connected.
Thus, it is this well-connected subgraph that will be seen as a ``skeleton'' to which all remaining vertices will attach via $d$ edges in the next $(1+o(1))n\log n$ stages.

The remainder of this stage focuses on the analysis of an exploration process $(\cA_s, \cP_s, \cN_s)_{s\ge 0}$ gradually revealing the random graph $G$ sampled from the coloured configuration model introduced above. 
In this exploration process, $\cA_s$ represents the set of \emph{active} vertices at step $s$ which are ``found'' by the process but we still do not know their entire neighbourhood in $G$, $\cP_s$ represents the set of \emph{passive} vertices whose entire neighbourhood has been revealed before step $s$ and $\cN_s$ represents the \emph{non-explored} vertices which were not yet ``found'' by the process. 
Our goal in the end of the exploration process is to find the \emph{$(d-1)$-core of $G$} (that is, the largest subgraph of $G$ with minimum degree at least $d-1$) without revealing its edges and ensure that only $o(n)$ vertices in this $(d-1)$-core have degree $d-1$.

%produce an induced subgraph of $G$ containing no vertices of degree less than $d-1$, $o(n)$ vertices of degree $d-1$, $n-o(n)$ vertices of degree $d$ and where none of its edges was yet explored by the exploration process (so its distribution is still that of a coloured configuration model).}
%Let $G$ be the random graph sampled from the coloured configuration model introduced above. 
%By iteratively deleting the vertices of degree at most $d-2$ from $G$, we construct a graph $H\subseteq G$ by analysing the following exploration process $(\cA_s, \cP_s, \cN_s)_{s\ge 0}$. 

We turn to the definition of the exploration process.
%Order the vertices of $G$ in an arbitrary way. 
In the beginning, we set $\cA_0 = \{v\in V(G): \deg(v)\le d-2\}$, $\cP_0 = \emptyset$ and $\cN_0 = V(G)\setminus \cA_0$. At step $s\ge 1$ of the process, look for a vertex in $\cA_{s-1}$ incident to at most $d-2$ edges that are still unrevealed throughout the exploration process; we denote this set of vertices by $\cA'_{s-1}$. 
If $\cA'_{s-1} = \emptyset$, terminate the exploration process and set $H = G[\cA_{s-1}\cup \cP_{s-1}]$. 
(In this case, by construction, $H$ is the $(d-1)$-core of $G$.)
Otherwise, pick the vertex $v$ in $\cA'_{s-1}$ with the smallest index (recall that $V(G) = \{w_1,\ldots,w_n\}$) and explore its neighbours, which amounts to setting
$$\cA_s = (\cA_{s-1}\setminus \{v\})\cup (N(v)\cap \cN_{s-1}),\; \cP_s = \cP_{s-1}\cup \{v\} \text{ and } \cN_s = \cN_{s-1}\setminus N(v).$$
Define $\tau$ as the unique integer when $\cA_{\tau}' = \emptyset$.
We note that the exploration process formally describes the usual way of obtaining the $(d-1)$-core of a graph by iteratively deleting vertices of degree at most $d-2$. 
However, since this core must remain a random graph in our case, we carefully keep track of the randomness exposed throughout these iterative deletions.

\begin{lemma}\label{lemma exp process}
The exploration process $(\cA_s, \cP_s, \cN_s)_{s\ge 0}$ a.a.s.\ terminates after at most $n/\omega$ steps.
\end{lemma}
\begin{proof}
First of all, Lemma~\ref{lemma non max deg} implies that $|\cA'_0| = |\cA_0|\le \omega^3 n/\log n$. 
Fix an integer $s\le n/\omega$. 
If~$s\le \tau-1$, then $|\cA'_{s+1}| \ge |\cA'_s|-1$ and equality holds unless the vertex $v$ explored at step $s$ connects (via an edge revealed at the current step) to some vertex in $\cA_s$ or to some vertex in $\cN_s$ of degree at most $d-1$.
Hence, the probability that the number of active vertices with at most $d-2$ unmatched edges decreases is at least 
\[1 - (d-2)\frac{|\cA_s|+|\{v\in V(G): \deg_G(v)\le d-1\}|}{|\cN_s|}\ge 1 - d\frac{dn/\omega + o(n/\omega)}{n - dn/\omega}\ge 1 - \frac{2d^2}{\omega}.\]
Hence, $(|\cA'_{\min(s,\tau)}|)_{s\ge 0}$ is dominated by the process $(X_{\min(s,\tau)})_{s\ge 0}$ where $(X_s)_{s\ge 0}$ is a random walk that makes a step $-1$ with probability $1 - 2d^2/\omega$ and step $+\,d$ with probability $2d^2/\omega$. We conclude that
$$\Prob(|\cA'_{\min(n/\omega,\tau)}|\ge 1) \le \Prob(X_{n/\omega}\ge 1)\le \Prob\left(|\{s\in [n/\omega]: X_s - X_{s-1} = d\}|\ge \frac{n}{2d\omega}\right) = o(1),$$
where the last equality follows from Markov's inequality and the fact that, on average, the random walk makes $O(n/\omega^2)$ steps $+\,d$ until step $n/\omega$.
\end{proof}

Now, define $G' = G[\cA_{\tau}\cup \cN_{\tau}]$.
The next two results tell us more about the structure of $G'$.

\begin{lemma}\label{lemma structural}
The graph $G'$ contains only vertices of degree $d-1$ and $d$. Moreover, the number of vertices of degree $d-1$ in $G'$ is at most $|\cA_{\tau}|+|\{v\in V(G): \deg_G(v) = d-1\}|$.
\end{lemma}
\begin{proof}
First of all, note that the vertices in $G'$ that do not have the same degrees in $G$ and in $G'$ are exactly $\cA_{\tau}$. Assuming the first statement of the lemma, this proves the second statement.

%Thus, the second statement of the lemma is implied by the previous observation and the first statement. 

We show the first statement.
On the one hand, all vertices of degree at most $d-2$ in $G$ are contained in $\cP_{\tau}$.
On the other hand, no edge goes between $\cP_{\tau}$ and $\cN_{\tau}$ and, since $\cA'_{\tau} = \emptyset$ by assumption, all vertices in $\cA_{\tau}$ send at least $d-1$ edges towards $\cA_{\tau}\cup \cN_{\tau}$.
%$L(H)\subseteq \cA_{\tau}$ and every edge in $H$ incident to a vertex in $L(H)$ has its other endvertex in $\cP_{\tau}$.
%Hence, by deleting $H\setminus L(H)$ from $G$, the vertices in $L(H)$ lose exactly 1 neighbour.
%On the other hand, all vertices in $L(H)$ have degree $d$ in $G$: indeed, $\{v\in V(G): \deg_G(v)\le d-2\}\subseteq \cP_{\tau}$ (so this set is disjoint from $L(H)$) and $\cA'_{\tau} = \emptyset$ so $L(H) = \cA_{\tau}$ and every vertex in this set must be incident to at least $d-1$ edges outside $H$.
We conclude that all vertices in $\cA_{\tau}$ have degree $d-1$ in $G'$, which concludes the proof.
\end{proof}

\begin{corollary}\label{cor exp process}
A.a.s., at the end of the exploration process, the graph $G'$ contains at least $n-n/\omega$ vertices and at most $(d+1)n/\omega$ vertices of degree $d-1$ (in $G'$).
\end{corollary}
\begin{proof}
The first statement is implied by Lemma~\ref{lemma exp process} (under which $|\cP_{\tau}|\le n/\omega$). 
The second statement follows by combining Lemma~\ref{lemma structural}, the fact that a.a.s.\ $|\{v\in V(G): \deg(v)\le d-1\}| \le \omega^3 n/\log n \ll n/\omega$ by Lemma~\ref{lemma non max deg} and the fact that a.a.s.\ $|\cA_{\tau}|\le d|\cP_{\tau}|+|\cA_0| \le dn/\omega + o(n/\omega)$ by Lemmas~\ref{lemma non max deg} and~\ref{lemma exp process}.
\end{proof}

\noindent
\textbf{Stage 3: analysis of $G'$ and boosting to a $d$-connected graph.} ($(1+o(1))n\log n/2$ rounds, $o(n)$ accepted edges)\\

\noindent
At this last stage, we first find a (multi-)graph $G''$ obtained from $G'$ by deleting only a few edges with the additional property that Builder can construct $G''$ in the original process up to identifying multiple edges. 
Then, we show that a.a.s.\ Builder can boost $G''$ to a $d$-connected graph within $(1+o(1))n\log n/2$ rounds.

To begin with, recall that the exploration process from Stage 2 does not reveal any edge in $G'$. 
Therefore, the graph $G'$ is distributed as a random (multi-)graph constructed according to the coloured configuration model. 
The next classic lemma is a technical result saying that the typical local structure of $G'$ is tree-like.

%We note a minor point which is nevertheless a bit subtle: while markings are independent for edges with different sets of endvertices, this is not the case for multiple edges 
%(if the first of a group of multiple edges is ignored due to repetition, each of the following is ignored as well). 

\begin{lemma}\label{lemma double edges}
For every integer $\ell\ge 1$, a.a.s.\ the number of cycles of length $\ell$ in $G'$ is at most $\log n$. 
%In particular, a.a.s.\ no double edge of $G'$ is deleted in the comparison with the original process. 
\end{lemma}
\begin{proof}
First of all, Corollary~\ref{cor exp process} implies that a.a.s.\ $\Delta(G') = d$ and $n-o(n)$ of the vertices in $G'$ have degree $d$. We condition on this event. 
Then, the expected number of $\ell$-cycles is bounded from above by $\tbinom{n}{\ell} \ell! d^{\ell} \left(1/(n - o(n))\right)^\ell \le 2d^\ell \ll \log n$ (where the factor $d^{\ell}$ comes from choosing the colours of the $d$ edges along the cycle) and the first statement follows from Markov's inequality. 
\end{proof}

Recall that one may couple the auxiliary process with the first (at most) $d n \log n/\omega$ rounds of the original process by simply omitting the rounds at which a repeated edge is proposed for a second, third, etc.\ time. 
We use this in combination with Lemma~\ref{lemma repeated}, which allows us to dominate the set of repeated edges in the auxiliary process by a random set including every edge of $G'$ independently with probability $p = \log n/n$. We say that edges belonging to the latter random set are \emph{marked} and denote by $G''$ the subgraph of $G'$ containing only the unmarked edges.
We note a minor point which is nevertheless a bit subtle: if at least one of several multiple edges between two verices in $G'$ is marked, each of the edges between these two endpoints cannot be used by Builder in the original process. 
However, conditionally on the a.a.s.\ event from Lemma~\ref{lemma double edges} with $\ell=2$, the probability of the above event is $O(p\cdot 2\log n) = o(1)$, so we ignore it.

%For the second statement, let us condition on the above event in the case $\ell = 2$. Then, the expected number of double edges that Builder cannot use in the original process is at most $dp\log n = o(1)$ and Markov's inequality.
%\end{proof}
%\noindent
%From this point on, we assume that all edges of $G'$ are marked independently of each other. We denote by $G''$ the subgraph of $G'$ containing only the unmarked edges.

Denote $n'' = |V(G'')|$. We recall that a.a.s.\ $n'' = n-o(n)$ and Corollary~\ref{cor exp process} shows that there are a.a.s.\ only $O(n/\omega)$ of these vertices of degree $d-1$. We condition on these events.

\begin{lemma}\label{lemma d ge 4}
For every $d\ge 4$, there are a.a.s.\ no sets $U\subseteq V(G'')$ of size more than $6$ such that $|N_{G''}(U)|\le d-1$. Moreover, the number of vertices of $G''$ in sets $U\subseteq V(G'')$ of size in $[6]$ that satisfy $|N_{G''}(U)|\le d-1$ is a.a.s.\ $o(n)$.
\end{lemma}
\begin{proof}
Condition on the a.a.s.\ events (ensured by Corollary~\ref{cor exp process}) that $n''\ge n-n/\omega$ and that there are at most $(d+1)n/\omega$ vertices of degree $d-1$ in $G'$. We divide the proof into two cases according to the size of $U$.

First, we fix a sufficiently small $\eps > 0$, $s\in [\eps n, n''/2]$, a vertex set $U\subseteq V(G'')$ of size $s$ and a vertex set $W\subseteq V(G'')\setminus U$ of size at most $(\log n)^2$.
We show that, for any choice of $U$ and $W$ as above, $N_{G'}(U)\not\subseteq W$.
To begin with, there are at least $ds - (d+1)n/\omega$ half-edges sticking out of $U$ in the coloured configuration model generating the graph $G'$.
To have that $N_{G'}(U)\subseteq W$, each of these half-edges has to connect with another half-edge sticking out of $U\cup W$ in the same colour.
Moreover, after $j-1$ of the edges in a given colour sticking out of $U$ have been formed, at least $n'' - 2(j-1) - (d+1)n/\omega\ge n-2(j-1)-(d+2)n/\omega$ half-edges in that same colour remain unmatched.
Letting $s_1, \ldots, s_d$ be the number of vertices in $U$ having a half-edge in colour respectively $1,\ldots,d$ sticking out of them, the probability that $N_{G''}(U)\subseteq W$ is at most
\[\prod_{i=1}^d \prod_{j=1}^{\lceil s_i/2\rceil} \left(\frac{s_i + (\log n)^2 - 2(j-1)}{n-2(j-1)-(d+2)n/\omega}\right) = \mathrm{e}^{o(n)}
\prod_{i=1}^d \binom{\lceil n/2-(d+2)n/2\omega\rceil}{\lceil s_i/2+(\log n)^2/2\rceil}^{-1} = \mathrm{e}^{o(n)} \binom{\lceil n/2\rceil}{\lceil s/2\rceil}^{-d},\]
where we used that $s_i = s - o(n)$ for all $i\in [d]$.
A union bound over $s\in [\eps n, n''/2]$, the sets $U$ of size $s$ and the sets $W$ of size $\ell\le (\log n)^2$ gives that the probability that $N_{G'}(U)\subseteq W$ is bounded from above by
\begin{equation}\label{eq:unionb}
\sum_{s=\eps n}^{n''} \sum_{\ell=1}^{(\log n)^2} \binom{n}{s} \binom{n-s}{\ell} \mathrm{e}^{o(n)} \binom{\lceil n/2\rceil}{\lceil s/2\rceil}^{-d}.    
\end{equation}
However, for every $s\in [\eps n, n''/2]$, we have that $\tbinom{n-s}{\ell} = \mathrm{e}^{o(n)}$ and
\[\binom{n}{s} \binom{\lceil n/2\rceil}{\lceil s/2\rceil}^{-d} = \mathrm{e}^{o(n)} \exp\left(s\log\left(\frac{n}{s}\right)+(n-s)\log\left(\frac{n}{n-s}\right) - \frac{ds}{2}\log\left(\frac{n}{s}\right) - \frac{d(n-s)}{2}\log\left(\frac{n}{n-s}\right)\right).\]
Since $d\ge 3$, there is a constant $c = c(\eps) > 0$ such that, for every $s\in [\eps n, n''/2]$, the above expression is bounded from above by $\mathrm{e}^{-cn+o(n)}$, which is sufficient to conclude that the expression in~\eqref{eq:unionb} tends to zero.
In particular, this means that a.a.s.\ every set $U$ of size $s\in [\eps n, n''/2]$ has neighbourhood of size at least $(\log n)^2$.
Moreover, Markov's inequality implies that a.a.s.\ there are at most $(\log n)^2/2$ marked edges, so every set $U$ as above has $\Omega((\log n)^2) > d-1$ neighbours in $G''$.

%While the numerator of each of the $d$ products above is dominated by 
%\[2^{\lceil s_i/2\rceil} \left\lceil \frac{s+(\log n)^2}{2}\right\rceil! = 2^{\lceil s_i/2\rceil} (s/2\mathrm{e})^{s/2+o(n)},\] 
%the denominator is at least 
%\[2^{\lceil s_i/2\rceil} \prod_{j=1}^{\lceil s_i/2\rceil} \left(\frac{n}{2}-\frac{(d+2)n}{2\omega} - (j-1)\right) = 2^{\lceil s_i/2\rceil} \frac{(n/2\mathrm{e})^n}{((n-s)/2\mathrm{e})^{n-s}} \mathrm{e}^{o(n)}.\]
%Moreover, for every $i\in [d]$, $s_i = s-o(n)$, so the nume

Now, fix $s\in [7, \eps n]$, a vertex set $U\subseteq V(G'')$ of size $s$ and a vertex set $W\subseteq V(G'')\setminus U$ of size at most $d-1$. 
Then, there are at least $(d-1)s$ half-edges sticking out of $U$ in the coloured configuration model generating the graph $G'$.
To have that $W = N_{G''}(U)$, each of these half-edges either has to connect with another half-edge sticking out of $U\cup W$ in the same colour or participate in a marked edge.
Moreover, after $i-1$ of the half-edges in a given colour sticking out of $U$ have been matched, at least $n'' - 2(i-1) - (d+1)n/\omega\ge n-2(i-1)-(d+2)n/\omega$ half-edges in that same colour remain unmatched.
With $s_1, \ldots, s_d$ as above, the probability that $N_{G''}(U) = W$ is at most
\[\prod_{i=1}^d \prod_{j=1}^{\lceil s_i/2\rceil} \left(\frac{s_i + (d-1) - (j-1)}{n-2(j-1)-(d+2)n/\omega} + p\right).\]
Moreover, one can easily check that, for all $i\in [d]$, $j\in [\lceil s_i/2\rceil]$ and large $n$,
\[\frac{s_i+(d-1)-(j-1)}{n-2(j-1)-(d+2)n/\omega}+p\le \frac{2(s+\log n)}{n}.\]
Using a union bound over all possible choices for $U$ and $W$ (including $\ell = |W|$) and the fact that $\lceil s_1/2\rceil + \ldots + \lceil s_d/2\rceil \ge (d-1)s/2$, the probability that such sets with $N_{G''}(U) = W$ exist is at most
%Then, using that every vertex in $G'$ is incident to at least $d-1$ edges, the expected number of vertices in sets $U$ of size $s$ such that $|N_{G''}(U)|\le d-1$ is bounded from above by
\begin{equation}\label{eq split}
\sum_{\ell = 0}^{d-1} \binom{n''}{s}\binom{n''-s}{\ell} \left(\frac{2(s + \log n)}{n}\right)^{(d-1)s/2} \le \binom{n}{s} n^{d-1} \left(\frac{2(s + \log n)}{n}\right)^{(d-1)s/2}.
\end{equation}
Moreover, for every $s\in [7, n''/2-1]$, 
\begin{align*}
\frac{\tbinom{n}{s} n^{d-1} \left(\tfrac{2(s + \log n)}{n}\right)^{(d-1)s/2}}{\tbinom{n}{s+1} n^{d-1} \left(\tfrac{2(s + 1 + \log n)}{n}\right)^{(d-1)(s+1)/2}} 
&=\; \frac{s+1}{n-s}\cdot \left(\frac{n}{2(s + 1 + \log n)}\right)^{(d-1)/2}\cdot \left(\frac{s + \log n}{s + 1 + \log n}\right)^{(d-1)s/2}.
\end{align*}
Moreover, as $d\ge 4$, the above expression is larger than 1 and uniformly bounded away from 1 over the interval $s\in [7, \eps n]$ for some sufficiently small but fixed $\eps > 0$.
As a result, the sum of the right hand side of~\eqref{eq split} over $s\in [7, \eps n]$ is of the same order as the term for $s=7$. 
Since $d\ge 4$, the latter is bounded by
\[\binom{n}{7} n^{d-1} \left(\frac{2(7 + \log n)}{n}\right)^{7(d-1)/2} = n^{d+6-7(d-1)/2-o(1)} = o(1).\]

Finally, we recall that there are $o(n)$ vertices of degree less than $d$ in $G'$. 
Moreover, a.a.s.\ there are at most $(\log n)^2 = o(n)$ marked edges, and $O(\log n) = o(n)$ vertices in cycles of length at most 14 by Lemma~\ref{lemma double edges}. 
In particular, a.a.s.\ the 7-th neighbourhood of all but $o(n)$ of the vertices in $G''$ is isomorphic to the 7-th neighbourhood of any vertex in a $d$-regular tree.
We condition on this event and show that none of these vertices is in a set $U$ of size $s\in [6]$ satisfying $|N_{G''}(U)|\le d-1$. 
Fix such a vertex $v$ and let $U'\subseteq U$ be the connected component of $v$ in $G''[U]$. 
Then, $U'$ has neighbourhood of size at least $d$: indeed, the vertex $v$ itself has $d$ neighbours and adding the vertices in $U'$ one by one so that the set of all added vertices remains connected at every step shows that the size of the neighbourhood cannot decrease, which completes the proof.
%is included in the 10-th neighbourhood of $v$ in $G''$, and since this neighbourhood comes from pruning a $d$-regular tree, we obtain that $|N_{G''}(U)|\ge |N_{G''}(v)| = d$.
\end{proof}

To conclude the proof in the case $d\ge 4$, we condition on the a.a.s.\ event from Lemma~\ref{lemma double edges} and the graph $G''$ satisfying the a.a.s.\ statement of Lemma~\ref{lemma d ge 4}. 
As already pointed out, by the coupling in Lemma~\ref{lemma repeated}, Builder has a strategy to construct $G''$ in $o(n\log n)$ rounds up to identifying multiple edges (of which there are at most $\log n$ pairs). 
At the same time, connecting every vertex $v\in V(G'')$ in a set $U$ with $|U|\le 6$ and $|N_{G''}(U)|\le d-1$ to $d+5$ new neighbours in $G''$ is sufficient to transform $G''$ into a $d$-connected graph. 
Indeed, for every such vertex $v$ and every set $U$ as above that contains it, the size of the neighbourhood of $U$ in this larger graph would be at least $(d+5) - (|U|-1)\ge (d+5) - 5 = d$. 
Using that there are only $o(n)$ such vertices, and there are a.a.s.\ $o(n)$ vertices in $G\setminus G'$ as well, combining Lemma~\ref{lemma eq union bound} (applied for $k = d+5$, $\mu = 4/\omega\ge 2|V(G\setminus G')|/n$ and the sets $S_i = V(G'')\setminus \{w_i\}$ for all $i\in [n]$) with Lemma~\ref{lem last step} finishes the proof in the case $d\ge 4$.

Before continuing with the proof for $d\in \{2,3\}$, note that there is a natural reason for the above approach to fail in this case. Indeed, when $d=2$, a configuration model with vertices of degrees 1 and 2 only is a union of paths and thus has plenty of connected components. Moreover, when $d=3$, the presence of many vertices of degree 2 ensures long paths, which in turn implies that there will a.a.s.\ exist sets $U$ of size $|U|\gg 1$ satisfying $|N_{G''}(U)| = 2$.

\section{\texorpdfstring{Proof for the case $d\in \{2,3\}$}{}}

As mentioned in the outline of the proof, we divide the construction into three major stages. At the first stage, we build many long vertex-disjoint paths. At the second stage, we merge almost all paths constructed up to now into a single long cycle. For $d = 2$, the third stage will rely on Lemmas~\ref{lemma eq union bound} and~\ref{lem last step}, while for $d=3$, a bit of additional work is needed .\\

\noindent
\textbf{Stage 1: building long paths.} ($o(n\log n)$ rounds, $n+o(n)$ accepted edges for $d=2$, $\tfrac{3}{4}n+o(n)$ accepted edges for $d=3$)\\

\noindent
In this stage, we use $o(n\log n)$ rounds to construct a set of paths of total length $n-o(n)$ for $d=2$, and $\tfrac{3}{4}n - o(n)$ for $d=3$. 
The proof is given for $d=2$; we point out a single minor modification in the case $d=3$ along the way.

Fix $t_1 = \omega^5 n$ and $N = n/2\omega^3$. Starting from $N$ vertices, we iteratively construct $N$ disjoint paths. 
Initially, each of these paths consists of one vertex. 
Then, at every round, if the proposed edge connects a vertex outside the $N$ paths to the last added vertex in some of the paths, Builder accepts the edge, otherwise they ignore it, see Figure~\ref{fig 1}.

\begin{figure}[ht]
\centering
\begin{tikzpicture}[line cap=round,line join=round,x=1cm,y=1cm]
\clip(-11.634891094552783,0.33897688139665544) rectangle (2.8505467623554837,7.430015832291529);
\draw [line width=0.1pt] (-10,4)-- (0.5,4);
\draw [line width=0.1pt] (-4.75,-20)-- (-4.75,20);

\draw [line width=0.8pt] (-2.5,5)-- (-2.5,6);
\draw [line width=0.8pt, opacity = 0.2] (-0.5,6)-- (-0.5,5);
\draw [line width=0.8pt] (-8,3)-- (-8,2);
\draw [line width=0.8pt] (-6,3)-- (-6,2);
\draw [line width=0.8pt] (-2.5,3)-- (-2.5,2);
\draw [line width=0.8pt] (-0.5,3)-- (-0.5,2);
\draw [line width=0.8pt, opacity = 0.2] (-8,3)-- (-7.5,2);
\draw [line width=0.8pt, opacity = 0.2] (-2.5,1)-- (-1.5,1);
\draw [line width=0.8pt, opacity = 0.2] (-8,5)-- (-8,6);
\begin{scriptsize}
\draw [fill=black] (-9,6) circle (1.5pt);
\draw [fill=black] (-8,6) circle (1.5pt);
\draw [fill=black] (-6.8,6) circle (0.5pt);
\draw [fill=black] (-7.2,6) circle (0.5pt);
\draw [fill=black] (-7,6) circle (0.5pt);
\draw [fill=black] (-6,6) circle (1.5pt);
\draw [fill=black] (-8,5) circle (1.5pt);
\draw [fill=black] (-9,6) circle (1.5pt);
\draw [fill=black] (-8,6) circle (1.5pt);
\draw [fill=black] (-6,6) circle (1.5pt);
\draw [fill=black] (-8,5) circle (1.5pt);
\draw [fill=black] (-9,6) circle (1.5pt);
\draw [fill=black] (-8,6) circle (1.5pt);
\draw [fill=black] (-6.8,6) circle (0.5pt);
\draw [fill=black] (-7.2,6) circle (0.5pt);
\draw [fill=black] (-7,6) circle (0.5pt);
\draw [fill=black] (-6,6) circle (1.5pt);
\draw [fill=black] (-8,5) circle (1.5pt);
\draw [fill=black] (-3.5,6) circle (1.5pt);
\draw [fill=black] (-2.5,6) circle (1.5pt);
\draw [fill=black] (-2.5,5) circle (1.5pt);
\draw [fill=black] (-0.5,6) circle (1.5pt);
\draw [fill=black] (-9,3) circle (1.5pt);
\draw [fill=black] (-8,3) circle (1.5pt);
\draw [fill=black] (-6,3) circle (1.5pt);
\draw [fill=black] (-3.5,3) circle (1.5pt);
\draw [fill=black] (-2.5,3) circle (1.5pt);
\draw [fill=black] (-0.5,3) circle (1.5pt);
\draw [fill=black] (-0.5,5) circle (1.5pt);
\draw [fill=black] (-8,2) circle (1.5pt);
\draw [fill=black] (-6,2) circle (1.5pt);
\draw [fill=black] (-2.5,2) circle (1.5pt);
\draw [fill=black] (-2.5,1) circle (1.5pt);
\draw [fill=black] (-0.5,2) circle (1.5pt);
\draw [fill=black] (-7.5,2) circle (1.5pt);
\draw [fill=black] (-1.5,6) circle (0.5pt);
\draw [fill=black] (-1.3,6) circle (0.5pt);
\draw [fill=black] (-1.7,6) circle (0.5pt);
\draw [fill=black] (-1.5,3) circle (0.5pt);
\draw [fill=black] (-1.3,3) circle (0.5pt);
\draw [fill=black] (-1.7,3) circle (0.5pt);
\draw [fill=black] (-7,3) circle (0.5pt);
\draw [fill=black] (-6.8,3) circle (0.5pt);
\draw [fill=black] (-7.2,3) circle (0.5pt);
\draw [fill=black] (-1.5,1) circle (1.5pt);

\draw [fill=black] (-7.5,6.5) node {\huge{1}};
\draw [fill=black] (-2,6.5) node {\huge{2}};
\draw [fill=black] (-2,3.5) node {\huge{4}};
\draw [fill=black] (-7.5,3.5) node {\huge{3}};
\end{scriptsize}
\end{tikzpicture}
\caption{The figure represents the first 4 rounds of the process. The initial $N$ vertices are ordered horizontally just below the round number. The transparent edges are the ones proposed at the given round. Thus, Builder accepts the first two proposed edges, ignores the third one since it is not incident to the last vertex that entered the first path, and ignores the fourth one since it is not incident to any of the vertices that are currently in the $N$ paths.}
\label{fig 1}
\end{figure}
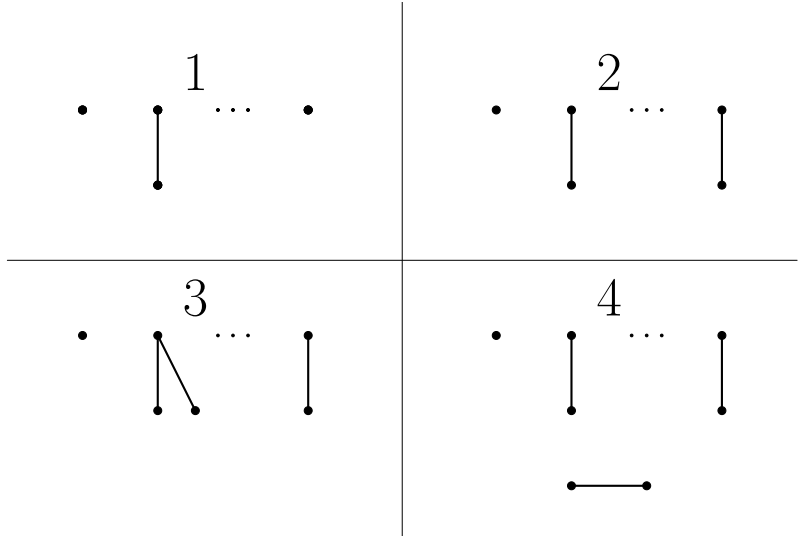

\begin{lemma}\label{claim 1}
After round $t_1$, the $N$ paths a.a.s. contain at least $n-n/\omega$ vertices. Moreover, at this point a.a.s. at least $(1-2/\omega) n$ of these vertices are in paths of length between $\omega^3$ and $3\omega^3$.
\end{lemma}

Note that, when $d=3$, instead of reaching round $t_1$ and then stop extending the paths, we stop when the total number of vertices in the $N$ paths reaches $\tfrac{3}{4}n - 1$. By Lemma~\ref{claim 1} a.a.s.\ this moment comes before round $t_1$.

\begin{proof}[Proof of Lemma~\ref{claim 1}]
Let $T$ denote the round when there remain exactly $n/\omega$ vertices outside the $N$ paths. Note that, at every round $s\in [0,T-1]$, the probability that one of the $N$ paths is extended is bounded from below by 
$$\left(\frac{n}{2\omega^3}\right) \left(\frac{n}{\omega}\right)\bigg{/}\binom{n}{2} \ge \frac{1}{\omega^4}.$$
Hence, the number of vertices $X_s$ in the $N$ paths after round $s\le T-1$ stochastically dominates the sum of $s$ independent Bernoulli random variables with success probability $1/\omega^4$. By the above comparison and Chernoff's inequality, 
$$\Prob(T > t_1) = \Prob(X_{t_1}\le n-n/\omega) \le \Prob(\mathrm{Bin}(t_1, 1/\omega^4)\le n) = o(1),$$
which proves the first statement.

Let us condition on the event $\{T \le t_1\}$. Denote by $Y_1$ the number of vertices in paths of length at most $\omega^3$, and by $Y_2$ the number of vertices in paths of length at least $3\omega^3$. We will show that $\mathbb E[Y_1+Y_2] = o(n/\omega)$, and then conclude by Markov's inequality.

Note that, at every round when some path is augmented, each of the $N$ paths has equal probability to be the augmented one. 
Thus, after $n-o(n)$ vertices have been included in the $N$ paths, the length of every path is a binomial random variable with parameters $n-o(n)$ and $1/N$. 
Since on average every path contains $(1-o(1))n/N = (2-o(1))\omega^3$ vertices after $t_1$ rounds, Chernoff's inequality again implies that one path contains at most $\omega^3$ vertices with probability $\exp(-\Omega(\omega^3))\le 1/\omega^3$. 
Hence, $\mathbb E Y_1\le \omega^3 N/\omega^3 = N = o(n/\omega)$.

Moreover, for every $i\ge 0$, Chernoff's inequality provides also that
$$\Prob(\mathrm{Bin}(n-n/\omega, 1/N)\ge 3\omega^3+i)\le \exp(-\Omega(\omega^3+i)),$$
so we conclude that
$$\mathbb E Y_2 = N \sum_{i\ge 0} (3\omega^3 + i)\Prob(\mathrm{Bin}(n-n/\omega, 1/N)\ge 3\omega^3+i)\le N \sum_{i\ge 0} (3\omega^3 + i) \exp(-\Omega(\omega^3+i)) = o(n/\omega).$$
Thus, using $\mathbb E [Y_1+Y_2] = o(n/\omega)$ and Markov's inequality concludes the proof of the second statement.
\end{proof}

\noindent
\textbf{Stage 2: merging the paths.} ($o(n\log n)$ rounds, $o(n)$ accepted edges)\\

\noindent
The reasoning in this stage is given for $d=2$. The adaptation to the case $d=3$ is immediate and requires no additional arguments.

To begin with, condition on the a.a.s.\ events from Lemma~\ref{claim 1} and let $M$ be the number of paths of length between $\omega^3$ and $3\omega^3$ after $t_1$ rounds. We call these paths \emph{typical}. At the second stage, we introduce an auxiliary directed graph constructed as follows.

Define $2M$ sets $(S'_i)_{i=1}^{M}$ and $(S''_i)_{i=1}^{M}$ where $S'_i$ contains the first $\omega^2$ vertices and $S''_i$ contains the last $\omega^2$ vertices in the $i$-th typical path. Also, define $t_2 = t_1 + \omega^3 n$ and let Builder accept the edge $e$ at round $s\in [t_1+1, t_2]$ if there are $i,j\in [M]$ such that $e$ connects a vertex in $S'_i$ to a vertex in $S''_j$. Now, the auxiliary directed graph $H$ with vertex set $[M]$ is constructed iteratively by adding the edge $ij$ if there is a round $s\in [t_1+1, t_2]$ at which Builder constructs an edge $S'_i$ and $S''_j$, see Figure~\ref{fig 2}. 
Note that, for two vertices $i,j$ in $H$, the edges $ij$ and $ji$ can participate simultaneously in $H$.

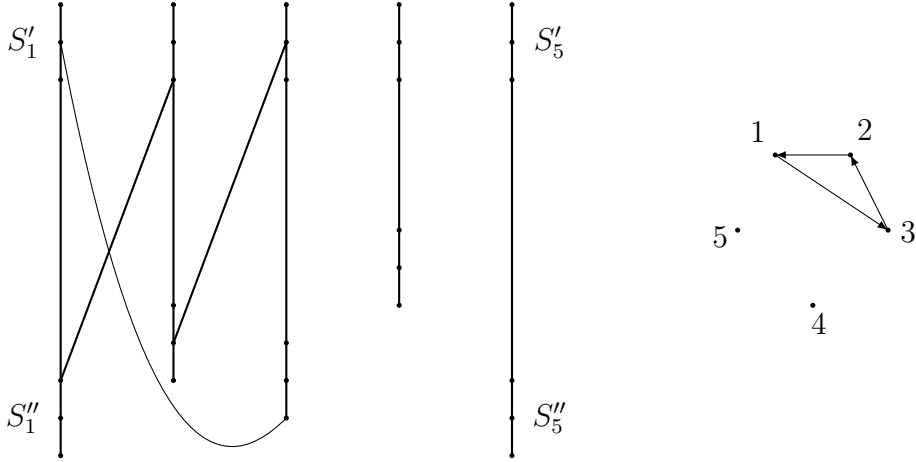
\begin{figure}
\centering
\begin{tikzpicture}[scale = 0.5,line cap=round,line join=round,x=1cm,y=1cm]
\clip(-10.14,-8.05) rectangle (20.42,6.91);
\draw [line width=0.8pt] (-5,6)-- (-5,-6);
\draw [line width=0.8pt] (-2,6)-- (-2,-4);
\draw [line width=0.8pt] (1,6)-- (1,-5);
\draw [line width=0.8pt] (4,-2)-- (4,6);
\draw [line width=0.8pt] (7,6)-- (7,-6);
\draw [line width=0.8pt] (1,5)-- (-2,-3);
\draw [line width=0.8pt] (-2,4)-- (-5,-4);
\draw [-latex] (17,0) -- (16,2);
\draw [-latex] (14,2) -- (17,0);
\draw [-latex] (16,2) -- (14,2);
\begin{scriptsize}
\draw [fill=black] (-5,6) circle (1.5pt);
\draw [fill=black] (-5,-6) circle (1.5pt);
\draw [fill=black] (-2,6) circle (1.5pt);
\draw [fill=black] (-2,-4) circle (1.5pt);
\draw [fill=black] (1,6) circle (1.5pt);
\draw [fill=black] (1,-5) circle (1.5pt);
\draw [fill=black] (4,-2) circle (1.5pt);
\draw [fill=black] (4,6) circle (1.5pt);
\draw [fill=black] (7,6) circle (1.5pt);
\draw [fill=black] (7,-6) circle (1.5pt);
\draw [fill=black] (-2,5) circle (1.5pt);
\draw [fill=black] (1,5) circle (1.5pt);
\draw [fill=black] (4,5) circle (1.5pt);
\draw [fill=black] (7,5) circle (1.5pt);
\draw [fill=black] (-5,5) circle (1.5pt);
\draw [fill=black] (-5,4) circle (1.5pt);
\draw [fill=black] (-2,4) circle (1.5pt);
\draw [fill=black] (1,4) circle (1.5pt);
\draw [fill=black] (4,4) circle (1.5pt);
\draw [fill=black] (7,4) circle (1.5pt);
\draw [fill=black] (-5,-4) circle (1.5pt);
\draw [fill=black] (-5,-5) circle (1.5pt);
\draw [fill=black] (-2,-2) circle (1.5pt);
\draw [fill=black] (-2,-3) circle (1.5pt);
\draw [fill=black] (1,-3) circle (1.5pt);
\draw [fill=black] (1,-4) circle (1.5pt);
\draw [fill=black] (4,0) circle (1.5pt);
\draw [fill=black] (4,-1) circle (1.5pt);
\draw [fill=black] (7,-4) circle (1.5pt);
\draw [fill=black] (7,-5) circle (1.5pt);
\draw [fill=black] (14,2) circle (1.5pt);
\draw [fill=black] (16,2) circle (1.5pt);
\draw [fill=black] (17,0) circle (1.5pt);
\draw [fill=black] (15,-2) circle (1.5pt);
\draw [fill=black] (13,0) circle (1.5pt);

\draw [fill=black] (-6,5) node {\large{$S'_1$}};
\draw [fill=black] (-6,-5) node {\large{$S''_1$}};

\draw [fill=black] (8,5) node {\large{$S'_5$}};
\draw [fill=black] (8,-5) node {\large{$S''_5$}};

\draw [fill=black] (13.54,2.63) node {\large{$1$}};
\draw [fill=black] (16.38,2.67) node {\large{$2$}};
\draw [fill=black] (17.54,-0.05) node {\large{$3$}};
\draw [fill=black] (15.16,-2.5) node {\large{$4$}};
\draw [fill=black] (12.54,-0.19) node {\large{$5$}};

\draw (-5,5) .. controls (-3,-5.5) and (-1,-7) .. (1,-5);
\end{scriptsize}
\end{tikzpicture}
\caption{The figure depicts the first 5 typical paths (on the left) as well as the graph, induced from $H$ by the 5 corresponding vertices (on the right). The sets $S'_1,\ldots, S'_5$ are depicted at the top while the sets $S''_1,\ldots, S''_5$ are put at the bottom. Then, the three edges between $S'_1$ and $S''_3$, $S'_3$ and $S''_2$, and $S'_2$ and $S''_1$ result in the edges $13$, $32$ and $21$ in $H$.}
\label{fig 2}
\end{figure}

\begin{lemma}\label{lemma long cycle}
After round $t_2$, the graph $H$ a.a.s.\ contains a directed cycle of length at least $M-M/\omega$.
\end{lemma}
\begin{proof}
First of all, let $\cE$ be the event that, for every pair $i, j\in [M]$, there are at most $\omega^4/2$ of all edges between $S'_i$ and $S''_j$ that have been proposed until round $t_1$. We show that $\cE$ holds a.a.s. Indeed, by Remark~\ref{rem comp} one may couple $G_{t_1}$ and $G(n,p)$ for $p = \log n/n$ so that a.a.s.\ $G_{t_1}\subseteq G(n,p)$. Moreover, the expected number of balanced bipartite graphs on $2\omega^2$ vertices and at least $\omega^4/2$ edges in $G(n,p)$ is bounded from above by
\begin{equation*}
    \binom{n}{2\omega^2} \binom{2\omega^2}{\omega^2} \binom{\omega^4}{\omega^4/2} \left(\frac{\log n}{n}\right)^{\omega^4/2} \le n^{2\omega^2} 2^{2\omega^2+\omega^4} \left(\frac{\log n}{n}\right)^{\omega^4/2} = o(1).
\end{equation*}
Since this is an upper bound for $\Prob(\overline{\cE})$, $\cE$ holds a.a.s.\ and we condition on this event.

Now, fix $i,j\in [M]$ and denote by $E_{i,j}$ the set of edges between $S'_i$ and $S''_j$ that were not proposed until round $t_1$, and by $\cE_{i,j}$ the event that at least one edge in $E_{i,j}$ was constructed at some round in the interval $[t_1+1, t_2]$. Then, the probability that $H$ contains the edge $ij$ is given by
\begin{equation}\label{eq E_i,j}
\begin{split}
\Prob(\cE_{i,j}) 
&=\; 1 - \prod_{s=t_1+1}^{t_2} \left(1-\frac{|E_{i,j}|}{n(n-1)/2-s+1}\right)\\
&=\; 1 - \exp\left(-(2+o(1))\frac{|E_{i,j}|\omega^3}{n}\right) = (2+o(1))\frac{|E_{i,j}|\omega^3}{n}.   
\end{split}
\end{equation}
Moreover, for every pair $i_1, j_1, i_2, j_2\in [M]$ with $(i_1,j_1)\neq (i_2,j_2)$, $\Prob(\cE_{i_1,j_1}\cap \cE_{i_2, j_2})$ can be expressed as
\begin{align*}
& \sum_{u=t_1+1}^{t_2}\prod_{s=t_1+1}^{u-1} \left(1-\frac{|E_{i_1, j_1}| + |E_{i_2, j_2}|}{n(n-1)/2-s+1}\right) \frac{|E_{i_1, j_1}|}{n(n-1)/2-u+1}\left(1 - \prod_{s = u+1}^{t_2}\left(1-\frac{|E_{i_2, j_2}|}{n(n-1)/2-s+1}\right)\right)\\
+\;
& \sum_{u=t_1+1}^{t_2}\prod_{s=t_1+1}^{u-1} \left(1-\frac{|E_{i_1, j_1}| + |E_{i_2, j_2}|}{n(n-1)/2-s+1}\right) \frac{|E_{i_2, j_2}|}{n(n-1)/2-u+1}\left(1 - \prod_{s = u+1}^{t_2}\left(1-\frac{|E_{i_1, j_1}|}{n(n-1)/2-s+1}\right)\right),
\end{align*}
where $u$ denotes the first moment when an edge between $S'_{i_1}$ and $S''_{j_1}$, or between $S'_{i_2}$ and $S''_{j_2}$, has been proposed to Builder. 
Moreover, the first expression computes the probability that this edge goes between $S'_{i_1}$ and $S''_{j_1}$ and at least one edge between $S'_{i_2}$ and $S''_{j_2}$ is proposed at some of the next $t_2-u$ rounds, while the second expression computes the probability of the other scenario.
Using that for every $u\in [t_1+1, t_2]$ and $\ell\in [2]$ we have 
$$\prod_{s=t_1+1}^{u-1} \left(1-\frac{|E_{i_1, j_1}| + |E_{i_2, j_2}|}{\tbinom{n}{2}-s+1}\right) = 1-o(1) \text{ and } \prod_{s=u+1}^{t_2} \left(1-\frac{|E_{i_\ell, j_\ell}|}{\tbinom{n}{2}-s+1}\right) = 1 - (1+o(1))\frac{2(t_2-u)|E_{i_\ell, j_\ell}|}{n^2},$$
the above expression rewrites
\begin{align*}
(1+o(1))\frac{8|E_{i_1, j_1}||E_{i_2, j_2}|}{n^4}\sum_{u=t_1+1}^{t_2} (t_2-u) = (1+o(1)) \frac{4(t_2-t_1)^2|E_{i_1, j_1}||E_{i_2, j_2}|}{n^4},
\end{align*}
which by~\eqref{eq E_i,j} is also equal to $(1+o(1))\Prob(\cE_{i_1,j_1})\Prob(\cE_{i_2, j_2})$. Thus,
\begin{equation*}
\mathrm{Var}\left[\sum_{i,j\in [M]} \ind_{\cE_{i,j}}\right] = o\left(\mathbb E\left[\sum_{i,j\in [M]} \ind_{\cE_{i,j}}\right]^2\right),
\end{equation*}
which means that a.a.s.\ $H$ contains 
$$(1+o(1))\mathbb E\left[\sum_{i,j\in [M]} \ind_{\cE_{i,j}}\right] \ge \frac{\omega^7 M^2}{n}\ge \frac{\omega^4 M}{4}$$
edges. 
We conclude by Theorem~1 from~\cite{KLS13} implying that the binomial random directed graph on $M$ vertices and with edge probability $q = \omega^3/M$ a.a.s.\ contains a directed cycle covering $M-M/\omega$ vertices, and the fact that $H$ a.a.s.\ dominates such a graph by Remark~\ref{rem comp}.
\end{proof}

\noindent
\textbf{Stage 3.1: completing the picture for $d=2$.} ($(1+o(1))n\log n/2$ rounds, $o(n)$ accepted edges)\\

\noindent
To conclude the proof for $d=2$, note that a cycle containing almost all vertices in $H$ implies that there is a cycle in $G_{t_2}$ that contains $n-o(n)$ vertices. 
Indeed, for every $s\ge 2$, the existence of the cycle $i_1, \ldots, i_s$ in $H$ implies that, for every $j\in [s]$, a vertex $u_j\in S'_{i_j}$ was connected by an edge to a vertex in $v_{j+1} = S''_{i_{j+1}}$ (indices seen modulo $s$). 
Hence, the edges $(u_j v_{j+1})_{j\in [s]}$ and the subpaths from Stage 1 connecting $v_j$ to $u_j$ form a cycle $C$ containing almost all vertices in the $s$ used paths. 
Hence, applying Lemma~\ref{lemma eq union bound} (for $k=2$ and $S_i$ the vertex set of the cycle $C$) and Lemma~\ref{lem last step} (for $H = C$, which is a $2$-connected graph) shows that $(1+o(1))n\log n/2$ rounds are sufficient to construct two edges from every vertex outside $C$ to $C$ itself, which finishes the proof in the case $d=2$.\\

\noindent
\textbf{Stage 3.2: building an almost spanning almost $3$-connected graph over the long cycle for $d=3$.} ($o(n\log n)$ rounds, $\tfrac{3}{4} n + o(n)$ accepted edges)\\

\noindent
Denote by $C$ the cycle containing $\tfrac{3}{4}n - o(n)$ vertices in the case $d=3$ and set $t_3 = t_2 + n\log n/\omega$, $t_4 = t_3 + n\log n/\omega$ and $t_5 = t_4 + n\log n/\omega$. 
Also, let us divide the cycle $C$ into three paths of equal (or almost equal) length whose vertex sets we denote by $V_1, V_2$ and $V_3$, and denote by $V_4$ a set of vertices outside the cycle satisfying $|V_4| = \max(|V_1|, |V_2|, |V_3|)$ (recall that $|V(C)|\le \tfrac{3}{4}n - 1$, so there are at least $\max(|V_1|, |V_2|, |V_3|)$ vertices outside $C$).

Now, during the rounds in the interval $[t_2+1, t_3]$ (respectively $[t_3+1, t_4]$ and $[t_4+1, t_5]$), we let Builder greedily construct a matching between $V_1$ (respectively $V_2$ and $V_3$) and $V_4$. 
Note that this way, most of the vertices in $C$ obtain one additional neighbour while most vertices outside $C$ connect to three vertices in $C$. 
Denote by $\widehat C$ the graph obtained by restricting $G_{t_5}$ to the union of $C$ and the vertices with three neighbours of $C$.

\begin{lemma}\label{lemma binomial set}
The set of vertices in $C$ of degree $2$ in $\widehat C$ is stochastically dominated by a binomial random subset of $V(C)$ in which every element is included independently with probability $\omega^4/\log n$.
\end{lemma}
\begin{proof}
By the same argument as in Lemma~\ref{lemma non max deg}, we know that there are a.a.s.\ at most $\omega^{3} |V(C)|/\log n$ vertices in $C$ that do not participate in the matching between $V_1$ (respectively $V_2$ and $V_3$) and $V_4$ constructed during rounds $[t_2+1, t_3]$ (respectively $[t_3+1, t_4]$ and $[t_4+1, t_5]$). 
As a consequence, there are a.a.s.\ at least $|V_4| - 3\omega^3 |V(C)|/\log n - 1$ vertices in $V_4$ that have three neighbours in $C$ after round $t_5$. We condition on this event as well as on the set $S$ of these vertices.

Then, for every $i\in [3]$, the subset of vertices in $V_i$ that remains unmatched to a vertex in $S$ is distributed uniformly among all subsets of $V_i$ of size $|V_i|-|S|\le 1+3\omega^3 |V(C)|/\log n$. 
Hence, by Lemma~\ref{lemma comp}, one may a.a.s.\ stochastically dominate each of them by a binomial random subset of $V_i$ in which every element is chosen with probability $2\tfrac{1+3 \omega^3 |V(C)|/\log n}{\min(|V_1|,|V_2|,|V_3|)} \le \omega^4/\log n$, as desired.
\end{proof}

\begin{lemma}\label{lemma paths}
The graph $\widehat C$ is $2$-connected. Moreover, all cutsets $U$ in $\widehat C$ of size $2$ satisfy that $U\subseteq C$ and $\widehat C\setminus U$ contains two connected components, one of which is a path in $C$ consisting of vertices of degree $2$ in $\widehat C$.
\end{lemma}
\begin{proof}
The first statement follows from Lemma~\ref{lem last step} for $H = C$ (which is a 2-connected graph) and $G = \widehat C$.

%and the fact that note that $\widehat C$ is obtained from $C$, which is a $2$-connected graph, by adding a set of vertices that all connect to $C$ by at least 2 edges.

Now, if $\widehat C = C$, the second statement holds trivially. Otherwise, let $\{v_1, v_2\}$ be a cutset in $\widehat C$. Then, since $C$ itself is a $2$-connected graph and all other vertices connect to $C$ by 3 edges, we must have that $\{v_1, v_2\}\subseteq V(C)$. 
Now, assume without loss of generality that $v_1\in V_1$ and $v_2\in V_1\cup V_2$. 
Suppose for contradiction that the path between $v_1$ and $v_2$ that is disjoint from $V_3$ contains a vertex $u$ of degree 3 in $\widehat C$. 
Then, there is a vertex in $V_3$ that has a common neighbour with $u$ in $\widehat C\setminus C$. 
Thus, all vertices in $C\setminus \{v_1, v_2\}$ are in the same connected component in $\widehat C\setminus \{v_1, v_2\}$, and since all vertices in $\widehat C\setminus C$ are connected to $C\setminus \{v_1, v_2\}$ by at least one edge, $C\setminus \{v_1, v_2\}$ must be a connected graph, which is a contradiction with the fact that $\{v_1, v_2\}$ is a cutset. This completes the proof.
\end{proof}

For every vertex $v\in V(C)$, denote by $P(v)$ the unique path in $C$ containing $v$ for which the first and the last vertex of $P(v)$ are of degree 3 in $\widehat C$, and all remaining vertices are of degree 2 in $\widehat C$.

\begin{corollary}\label{cor paths}
The following holds conditionally on the event $\widehat C\setminus C\neq \emptyset$: every graph $\widetilde C$ constructed from $\widehat C$ by connecting (by an edge) every vertex $v\in V(C)$ of degree $2$ in $\widehat C$ to a vertex in $\widehat C\setminus P(v)$, is $3$-connected.
\end{corollary}
\begin{proof}
Suppose for contradiction that this is not the case. 
Then, there is a cutset $\{v_1, v_2\}$ of $\widetilde C$ (which is also a cutset of $\widehat C$). 
As in the proof of Lemma~\ref{lemma paths}, assume without loss of generality that $v_1\in V_1$ and $v_2\in V_1\cup V_2$. Then, the path between $v_1$ and $v_2$ that is disjoint from $V_3$ contains only vertices of degree 2 in $\widehat C$. 
Let $v$ be one vertex in this path. 
Then, by Lemma~\ref{lemma paths}, $v_1, v_2$ belong to $P(v)$. 
Moreover, since $v$ connects by an edge to $\widehat C\setminus P(v)$ (which itself is a connected graph because all vertices in $\widehat C\setminus C$ are connected by an edge to $V_3$), we conclude that $\widetilde C\setminus \{v_1, v_2\}$ is also connected graph, which finishes the proof of the lemma.
\end{proof}

Finally, we complete the proof for $d=3$ by combining Lemmas~\ref{lemma eq union bound} and~\ref{lem last step} (used in the same way as in the case $d=2$) with the following lemma.

\begin{lemma}
A.a.s.\ Builder has a strategy to construct a $3$-connected supergraph of $\widehat C$ by accepting $o(n)$ additional edges and waiting for $(1+o(1))n\log n/2$ more rounds.
\end{lemma}
\begin{proof}
Recall that, by Lemma~\ref{lemma binomial set}, a.a.s.\ one can stochastically dominate the set $\{v\in V(C): \deg_{\widehat C}(v)=2\}$ by a binomial random subset of $V(C)$ with probability $\omega^4/\log n$. We work in the binomial model. 
Then, by Lemma~\ref{lemma binomial set} and Markov's inequality, there are at a.a.s.\ at most $\omega^5 n/\log n$ vertices of degree 2 in $\widehat C$. 
Moreover, for every vertex $v$, $\Prob(|P(v)|\ge \log n)\le (\omega^4/\log n)^{\log n} = o(n^{-1})$, so by a union bound a.a.s.\ $|P(v)|\le \log n$ for every $v\in V(C)$ of degree 2 in $\widehat C$. 

Let us condition on all of the above a.a.s.\ events. 
Then, the statement of the lemma is an application of Lemma~\ref{lemma eq union bound} allowing to construct edges between the vertices $v\in \{w_i\in V(C): \deg_{\widehat C}(w_i) = 2\}$ to the sets $S_i = V(\widehat C\setminus P(v))$ and $v\in V(G\setminus\widehat C)$ to the sets $S_i = V(\widehat C)$.
\end{proof}

The proof of Theorem~\ref{thm main} is completed.\\

\noindent
\textbf{Note added.} After the completion of the current note, I was informed that the following asymptotic version of Conjecture~\ref{conj main} was also proven independently by the authors of~\cite{FKM22}. The proof is to appear shortly in an updated version of~\cite{FKM22}.

\begin{theorem}
For every $\eps > 0$ there is $d_0 = d_0(\eps)$ such that for every $d\ge d_0$, the following holds: 
if $t\ge (1 + \eps)n \log n/2$ and $b\ge  (1 + \eps)dn/2$, then there exists a $(t, b)$-strategy of
Builder such that $G_t$ is a.a.s.\ $d$-connected.
\end{theorem}

\noindent
\textbf{Acknowledgements.} The author is grateful to Ivailo Hartarsky and to the two anonymous referees for many useful comments and suggestions.

\small
\bibliographystyle{plain}

\normalsize

\end{document}